\documentclass[reqno]{amsart}

\usepackage{graphicx} 
 
\usepackage{hyperref}

\hypersetup{
colorlinks=true,
linkcolor=blue,
anchorcolor=blue,
citecolor=blue
}

\usepackage{amssymb}
\usepackage{amsfonts}
\usepackage{amsmath}
\usepackage[utf8]{inputenc}
    
\newtheorem{theorem}{Theorem}[section]
\newtheorem{lemma}[theorem]{Lemma}
\newtheorem{proposition}[theorem]{Proposition}
\newtheorem{corollary}[theorem]{Corollary}

\theoremstyle{definition}
\newtheorem{definition}[theorem]{Definition}

\newtheorem{remark}[theorem]{Remark}

\numberwithin{equation}{section}

\newcommand{\R}{\mathbb{R}}

\newcommand{\Rn}{{\mathbb{R}^n}}

\newcommand{\Ha}{\mathcal{H}}
\newcommand{\M}{\mathcal{M}}

\newcommand{\ve}{\varepsilon}

\newcommand{\esssup}{\operatornamewithlimits{ \relax {esssup }}}

\def\loc{\qopname\relax o{loc}}

\def\dist{\qopname\relax o{dist}}

\def\phi{\varphi}

\renewcommand{\epsilon}{\varepsilon}

\begin{document}

\title{On Choquet integrals and Sobolev type inequalities}

\author{Petteri Harjulehto}
\address{Petteri Harjulehto, Department of Mathematics and Statistics, FI-00014 University of Helsinki,  Helsinki, Finland}
\email{petteri.harjulehto@helsinki.fi}

\author{Ritva Hurri-Syrj\"anen}
\address{Ritva Hurri-Syrj\"anen, Department of Mathematics and Statistics,
FI-000 University of Helsinki, Helsinki, Finland}
\email{ritva.hurri-syrjanen@helsinki.fi}





\begin{abstract}
We consider integrals in the sense of Choquet with respect to the $\delta$-dimensional Hausdorff content 
for  continuously differentiable functions defined on open, connected sets  in the Euclidean $n$-space, $n\geq 2$,
$0<\delta\le n$.  In particular, 
for these functions we prove Sobolev  inequalities 
in the limiting case $p=\delta /n$ and  in the case $p>\delta$, here $p$ is the integrability exponent of the absolute value of the gradient of any given function.
The results
complement
previously known Poincar\'e-Sobolev   and Morrey inequalities. 
\end{abstract}

\maketitle


\section{Introduction}

We study embeddings of continuously differentiable functions  $u$ defined on open, connected sets  $\Omega$ 
with some regularity  in $\Rn$, $n\geq 2$. Let $\delta \in (0,n]$ be given. Our assumption is that  the Choquet integral with respect to the  
$\delta$-dimensional Hausdorff content 
of the absolute  value  of  the gradient
to the power $p$ with some $p\geq \delta /n$
 is finite, that is
\[
\int_{\Omega} |\nabla u(x)|^{p} \, d \Ha^{\delta}_\infty = \int_0^\infty \Ha^{\delta}_\infty\big(\{x \in \Omega : |\nabla u(x) |^p>t\}\big) \, dt < \infty\,.
\]
For the values of $p$ there are three different cases
$\frac{\delta}{n} \le p < \delta$, $p=\delta$, and $p> \delta$ corresponding to the embeddings.
Whenever the dimension of the Hausdorff content $\delta$ equals  to the dimension of the Euclidean space  $\Rn$,  some of our results recover  the corresponding  
known classical 
inequalities.

The case $\frac{\delta}{n} < p < \delta$ was considered  in  \cite[Theorem 3.7]{HH-S_JFA} 
 and the case $\frac{\delta}{n} <p <\infty$ in \cite[Theorem 3.2]{HH-S_JFA}.
Fractional Sobolev spaces embeddings were proved  in \cite[Theorems 1.3, 1.5-1.7]{PonceSpector20}. 
The case $p=\delta$ was considered in 
\cite[Theorems 1.3, 1.5, Corollaries 1.4, 1.6]{MartinezSpector21}
and  \cite[Theorems 1.1, 5.9, Corollary 1.2]{HH-S_AAG}.
We recall also Sobolev embeddings related to Bessel and Riesz capacities in \cite[Theorems 1.4, 1.5]{OoiPhuc22}.
The present work can be seen as a complement to  the previous results.
Our main concern in the present paper is  the limiting case $p=\delta /n$ and the case $p>\delta$.
We recall the basic properties of Hausdorff content and Choquet integrals in Section \ref{Section:Def} trying to make the text as self contained as possible.
Poincar\'e-Sobolev inequalities are studied in Section \ref{Section:PS}.
One of the results there
is the following theorem.

\begin{theorem}\label{cor:SP-2d-epsilon}
Suppose that  $\Omega$ is a bounded  $(\alpha, \beta)$-John domain in $\Rn$, $n\geq 2$.
If $\delta \in (0,n]$ is fixed, then the inequality
\begin{equation*}
\inf_{b \in \R}  \Big(\int_\Omega |u(x) -b|^{\frac{\delta}{n-1}} d \Ha^{\delta}_\infty \Big)^
{\frac{n-1}{\delta}} \le c \Big(\int_{\Omega} |\nabla u(x)|^{\frac{\delta}{n}} \, d \Ha^{\delta- \ve}_\infty \Big)^{\frac{n}{\delta}}
\end{equation*}
holds for all $\ve \in(0, \delta(1-\frac1n))$ and
for all $u \in C^1(\Omega)$.
 Here $c$ is a constant which depends only on $\alpha, \beta, \delta, \ve, n,$ and
$\Ha^{\delta}_\infty(\Omega)$.
\end{theorem}

We consider the case $p=\delta$ in Section \ref{Section:Trudinger}, slightly improving a previous Trudinger inequality 
with Choquet integrals for John domains, Theorem \ref{thm:main-p=n}.
We give corresponding results  of Sections \ref{Section:PS} and \ref{Section:Trudinger}   for  continuously differentiable functions with compact support defined on bounded domains, and do so in Section \ref{Section:C^1_0}. One on the results there is the following weak-type estimate.

\begin{corollary}\label{weaklimit-0-0}
Suppose that  $\Omega$ is a bounded domain in $\Rn$ and $\delta \in (0, n]$ is given.
If 
$u \in C^1_0(\Omega)$, then 
for every $t>0$ 
\begin{equation*}
 \Ha_\infty^{\delta}\Bigl(\{x\in\Omega :\vert u(x)\vert >t\}\Bigr)^{\frac{n-1}{\delta}}
\le 
\frac{c}{t}
 \Big(\int_{\Omega}\vert \nabla u(x)\vert^{\frac{\delta}{n}}\,d\Ha_\infty^{\delta} \Big)^{\frac{n}{\delta}}\,
\end{equation*}
where  $c$  is a constant  which depends only on $n$ and $\delta$.
\end{corollary}

The case $p>\delta$ is considered in Section \ref{Section:p>delta}, where we prove Morrey inequalities.

{


\section{Definitions and notation}\label{Section:Def}
We are working on the  Euclidean $n$-space $\Rn$, $n\geq 2$.
An  open $n$-dimensional  ball  $B^n(x,r)$ centred at $x$ with radius $r>0$ is written as $B(x,r)$.

We recall the definition of Hausdorff content of a set $E$ in $\Rn$,  \cite[2.10.1, p.~169]{Federer}, \cite[Chapter 3]{Adams15}.

\begin{definition}[Hausdorff content]\label{defn:Haus}
Let $E$ be a set in $\Rn$, $n \ge 2$. Suppose that
$\delta \in (0, n]$.
The $\delta$-Hausdorff content of $E$ is defined by
\begin{equation*}
\Ha_\infty^{\delta} (E) := \inf \bigg\{ \sum_{i=1}^\infty r_i^{\delta}: E \subset \bigcup_{i=1}^\infty B(x_i, r_i)\,, r_i>0\bigg\}\,,
\end{equation*}
where the infimum is taken over all  countable (or finite) ball coverings of $E$.
\end{definition}
The terms the  $\delta$-Hausdorff capacity and
 and Hausdorff content of $E$ of dimension $\delta$ are also used.
 Recently, properties of Hausdorff content  have been studied extensively,
  for example \cite{DafniXiao04},
\cite{YangYuan08}, \cite{Tang},  \cite{SaitoTanakaWatanabe2016}, \cite{Liu16}, 
\cite{SaitoTanakaWatanabe2019}, \cite{Saito2019}, \cite{SaitoTanaka2022}, \cite{PonceSpector23},
\cite{ChenSpector}.

Let $0 <\delta\le n$ and $0<\rho <\infty$. 
We  recall the definition of the $\delta$-dimensional  Hausdorff measure for $E \subset \Rn$,
\[
\Ha^\delta (E) := \lim_{\rho \to 0^+}  \inf \bigg\{ \sum_{i=1}^\infty r_i^{\delta}: E \subset \bigcup_{i=1}^\infty B(x_i, r_i) \text{ and } r_i \le \rho \text{ for all } i\bigg\},
\]
where the infimum is taken over 
 all such countable  (or finite) ball coverings of $E$ that the radius of a  ball is at most $\rho$.
Thus  there are finite positive constants $c_1(n)$ and $c_2(n)$ such that $$c_1(n)\Ha^n(E) \le |E| \le c_2(n)\Ha^n(E)$$ for all Lebesgue measurable 
 sets $E$ in $\Rn$.
We recall the following property.
There exists a constant $c(n)>0$ such that  for all  sets $E$ in $\Rn$ the inequalities
\begin{equation}\label{content_measure}
\Ha_\infty^n(E) \le \Ha^n(E)\le c(n)\Ha_\infty^n(E),
\end{equation} 
hold
\cite[Proposition 2.5]{HH-S_JFA}.

We recall the definition of Choquet integrals over subsets of  $\Rn$  with respect to the Hausdorff content $\Ha_\infty^{\delta}$ \cite{Cho53}.
\begin{definition}[Choquet integral]\label{ChoquetIntegral}
Let $\Omega$ in $\Rn$  be an open  set, $f:\Omega\to [0,\infty]$  a function defined on $\Omega$ with values on $[0,\infty]$,  and
$\delta\in (0,n]$. The integral in the sense of Choquet with respect to Hausdorff content is defined by
\begin{equation}\label{IntegralDef}
\int_\Omega f(x) \, d \Ha^{\delta}_\infty := \int_0^\infty \Ha^{\delta}_\infty\big(\{x \in \Omega : f(x)>t\}\big) \, dt.
\end{equation}
\end{definition}
Since  $\Ha^{\delta}_\infty$ is monotone,
 for  every  function $f:\Omega \to [0, \infty]$
the corresponding distribution function
$t \mapsto \Ha^{\delta}_\infty\big(\{x \in \Omega : f(x)>t\}\big)$ 
is decreasing with respect to $t$ and 
hence also measurable
 with respect to Lebesgue measure. Thus, the integral on the right-hand side of (\ref{IntegralDef})
is well defined as a Lebesgue integral. 
Although the Choquet  integral is well defined for non-measurable functions we study here only measurable functions. 
We have a work in progress for non-measurable functions, too.
For the properties of Choquet integrals we refer to
\cite{Adams86}, 
\cite{Adams98}, \cite{Denneberg94}, \cite{OV}, \cite{DafniXiao04},  \cite{AX2012}, \cite{Adams15}, \cite{Kawabe19},
\cite{HH-S_AAG}, \cite{HH-S_JFA}, \cite{PonceSpector23}.

We recall here also the definition of the fractional Hardy-Littlewood maximal operator.
Let $f \in L^1_{\loc}(\Rn)$.  If $\kappa \in [0, n)$, the  fractional  centred Hardy-Littlewood maximal function of $f$ is defined by
\[
\M_\kappa f(x) := \sup_{r>0} r^{\kappa -n} \int_{B(x, r)} |f(y)| \, dy
\]
for each $x\in\Rn$.
The  non-fractional maximal function $\M_0 f$ is written as $\M f$.

The  maximal operator 
 $\M f$  is bounded in the sense of Choquet with respect to Hausdorff content when $p> \frac{\delta}{n}$. The maximal operator satisfies a weak-type estimate if $p= \frac{\delta}n$.
We refer to  \cite[Theorem A]{Adams86}, \cite[Theorem 7.(a)]{Adams98},  and \cite[Theorem]{OV}. 

Let us denote the $\eta$-order Riesz potential, $\eta \in(0, n)$ by
\[
I_\eta f(x) := \int_\Rn \frac{|f(y)|}{|x-y|^{n-\eta}} \, dy
\]
for  functions $f \in L^1_{\loc} (\Omega)$ 
 with an agreement that $f$ is defined to be zero outside $\Omega$.

We recall the following inequalities:
If $\Omega$ is an open set in  $\Rn$ and $0 < \delta \le n$, then
 there exists a constant $c(n)$ such that the inequalities
\begin{equation}\label{basic}
\frac{1}{c(n)}  \int_\Omega |f(x)|  \, d \Ha^{ n}_\infty \le \int_\Omega |f(x)| \, dx \le c(n)  \int_\Omega |f(x)|  \, d \Ha^{ n}_\infty
\end{equation}
are valid 
for all measurable  functions $f: \Omega \to [-\infty, \infty]$.
Indeed, the inequalities (\ref{basic}) follow from (\ref{content_measure}) and the isodiametric inequality.

Validity of  inequalities where integrals are Choquet integrals with respect to Hausdorff content have been studied in
\cite{Adams86}, \cite{Adams98},
\cite{PonceSpector20}, \cite{MartinezSpector21}, \cite{HH-S_AAG}, \cite{OoiPhuc22}, \cite{HH-S_JFA}.
We refer to
\cite{PonceSpector23} where also other  monotone set functions than Hausdorff content are considered and
results on Choquet integrals with minimal assumptions on these set functions  
are proved.

We state and prove our next proposition, although after finishing our results we found out 
 that a similar proposition was  just recently  proved also  in \cite[Proposition 2.3]{ChenOS_pp}.

\begin{proposition}\label{GeneralizationOV}
 Let $\Omega$ be an open subset in $\Rn$ and  $0 < \delta_1 < \delta_2 \le n$.
 Then the  inequality
\begin{equation*}
\biggl(\int_\Omega |f(x)|\, d \Ha_{\infty}^{\delta_2} \biggr)^{1/{\delta_2}}
\le 
\biggl(
 \frac{\delta_2}{\delta_1}\biggr)^{1/\delta_2} \biggl(\int_\Omega |f(x)|^{\frac{\delta_1}{\delta_2}} \, d \Ha^{\delta_1}_\infty \biggr)^{1/\delta_1}
\end{equation*}
holds for all functions $f: \Omega \to [-\infty, \infty]$.
\end{proposition}

\begin{proof}
Since $\delta_1 <\delta_2$, the function $t\mapsto t^{\delta_1/\delta_2}$ is concave on $[0,\infty )$.
Hence for all $r_i\geq 0$
the inequality
\begin{equation*}
\Big(\sum r_i^{\delta_2}\Big)^{\delta_1/\delta_2}\le \sum r_i^{\delta_1}
\end{equation*}
holds.  Thus, if  $E$ is a set in $\Rn$, then by the definition of the Hausdorff content 
\begin{equation*}
\Big(\Ha_{\infty}^{\delta_2}(E)\Big)^{1/\delta_2}\le\Big(\Ha_{\infty}^{\delta_1}(E)\Big)^{1/\delta_1}\,.
\end{equation*}
The changing of the variables, $t= s^{\delta_2/\delta_1}$, gives
\[
\begin{split}
\int_{\Omega}\vert f(x)\vert\,d\Ha_{\infty}^{\delta_2}
&= \int_0^\infty \Ha^{\delta_2}_\infty\Big(\{x: |f(x)|>t\} \Big) \, dt\\
&=   \frac{\delta_2}{\delta_1} \int_0^\infty 
s^{\frac{\delta_2}{\delta_1}-1} \,  \Ha^{\delta_2}_\infty\Big(\{x: |f(x)|>s^{\delta_2/\delta_1}\} \Big)  \, ds \\
&\le   \frac{\delta_2}{\delta_1} \int_0^\infty 
s^{\frac{\delta_2}{\delta_1}-1} \, \Ha^{\delta_1}_\infty\Big(\{x: \vert f(x)\vert^{\delta_1/\delta_2}>s\} \Big)^{\frac{\delta_2}{\delta_1}}  \, ds \,.
\end{split}
\]
 The function  $h(t) := \Ha^{\delta_1}_{\infty}(\{x: |f(x)|^{\delta_1/\delta_2}>t\}$ is decreasing. Hence, we obtain
\[
u h(u) = \int_{0}^u h(u) \, ds \le \int_{0}^u h(s) \, ds  \le \int_0^\infty h(s) \,ds.
\]
 Thus, combining the estimates implies
\[
\begin{split}
&\int_{\Omega}\vert f(x)\vert\,d\Ha_{\infty}^{\delta_2}\\
&\le \frac{\delta_2}{\delta_1} \int_0^\infty \bigg(s \, \Ha^{\delta_1}_\infty\Big(\{x: \vert f(x)\vert^{\delta_1/\delta_2}>s\} \Big)\bigg)^{\frac{\delta_2} {\delta_1}-1} \Ha^{\delta_1}_\infty\Big(\{x: \vert f(x)\vert^{\delta_1/\delta_2}>s\} \Big)\, ds \\
&\le  \frac{\delta_2}{\delta_1} \Big( \int_0^\infty h(s) \,ds \Big)^{\frac{\delta_2}{\delta_1} -1} \int_0^\infty h(t)\, dt\\
&\le  \frac{\delta_2}{\delta_1} \Big(  \int_0^\infty h(s) \,ds \Big)^{\frac{\delta_2}{\delta_1}} 
=  \frac{\delta_2}{\delta_1}\Big(\int_\Omega |f(x)|^{\frac{\delta_1}{\delta_2}} \, d \Ha^{\delta_1}_\infty \Big)^{\frac{\delta_2}{\delta_1}}.  \qedhere
\end{split}
\]
\end{proof}

Before going further we recall the definition of  bounded John domains.
An inner radius and outer radius  notion for domains was introduced by F. John  in \cite{John}. Later, domains with  this property were renamed John domains
\cite{MartioSarvas}. It has turn out that this notion defines extremely well  properties of  domains which are essentially like balls in  the potential theory aspect.

\begin{definition}\label{bounded-john}
Suppose that  $\Omega$ is a bounded domain in $ \Rn$, $n\geq 2$. The domain $\Omega$ is an $(\alpha, \beta)$-John domain if there exist constants $0< \alpha \le \beta<\infty$ and a point $x_0 \in \Omega$ such that each point $x\in \Omega$ can be joined to $x_0$ by a rectifiable curve $\gamma_x:[0,\ell(\gamma_x)] \to \Omega$, parametrised by its arc length, such that $\gamma_x(0) = x$, $\gamma_x(\ell(\gamma_x)) = x_0$, $\ell(\gamma_x)\leq \beta\,,$ and
\[
 \dist\big(\gamma_x(t), \partial \Omega \big) 
\geq  \frac{\alpha}{\beta}t
\quad \text{for all} \quad t\in[0, \ell(\gamma_x)  ].
\]
\end{definition}

The point $x_0$ is called a John centre.
Whenever $\Omega$  is a bounded John domain with respect to some point $x_0$,  it is  also a John domain with respect to any other point in $\Omega$.
Constants depend on the John centre though.
Examples of John domains are convex domains and domains with Lipschitz boundary, but also domains with fractal boundaries such as the von Koch snow flake.
Cuspidal domains like $s$-cusps with $s>1$ are not allowed, that is outward spires for domains  are not allowed.
Poincare-Sobolev inequalities with Choquet integrals  for these domains are considered in \cite{HH-S_Proc2022}.

We remark that the Euclidean $n$-space is an example of unbounded John domains.  For a definition of unbounded John domains and examples we refer to  \cite[Definition 2.1]{HSV}.
In particular, Morrey inequalities with
 Choquet integrals over the whole $\Rn$  are given in Theorems \ref{Morrey_first}, \ref{Morrey_second}, and Corollary \ref{Morrey_third}.


\section{Poincar\'e-Sobolev inequalities in the case $\delta /n \le p< \delta$}
\label{Section:PS}

We recall the following Poincar\'e-Sobolev inequality from \cite[Theorem 3.7]{HH-S_JFA}.

\begin{theorem}\cite[Theorem 3.7]{HH-S_JFA}\label{thm:Poincare}
Suppose that  $\Omega$ is a bounded  $(\alpha, \beta)$-John domain in $\Rn$.
Let   $\delta \in (0, n]$ and $p \in (\delta/n, \delta )$.
Let $\kappa \in [0, 1)$. Then there exists a constant $c$ depending only on $n$, $\delta$, $\kappa$, $p$, 
 and John constants $\alpha$ and $\beta$ such that  
\begin{equation*}
 \inf_{b \in \R}  \Big(\int_\Omega |u(x) -b|^{\frac{p(\delta- \kappa p) }{\delta-p}} d \Ha^{\delta -\kappa p}_\infty \Big)^{\frac{\delta-p}{p(\delta- \kappa p)}}
\le c \Big(\int_{\Omega} |\nabla u(x)|^{p} \, d \Ha^{\delta}_\infty \Big)^{\frac{1}{p}}
\end{equation*}
for all $u \in C^1(\Omega)$.
\end{theorem}

Choosing $p= \frac{\delta}{n-\ve}$ and $\kappa =0$} in 
Theorem \ref{thm:Poincare}  implies the following corollary.

\begin{corollary}\label{epsilon}
Suppose that  $\Omega$ is a bounded  $(\alpha, \beta)$-John domain in $\Rn$, $n\geq 2$.
If $\delta \in (0, n]$
then  for any fixed $\ve \in (0, n-1)$ there exists a constant $c$ such that  the inequality
\begin{equation*}
\inf_{b \in \R}  \Big(\int_\Omega |u(x) -b|^{\frac{\delta}{n-1-\ve}} d \Ha^{\delta}_\infty \Big)^{\frac{n-1-\ve }{\delta}}
\le c
\Big(\int_{\Omega} |\nabla u(x)|^{\frac{\delta}{n-\ve}} \, d \Ha^{\delta}_\infty \Big)^{\frac{n-\ve}{\delta}}
\end{equation*}
holds for all $u\in C^1(\Omega )$; the constant $c$ is independent of $u$.
\end{corollary}

Let $b\in\R$ be given.
If  $\ve\in (0, n-1)$ is  given, then H\"older's inequality with $\Big(\frac{n-1}{n-1-\ve}, \frac{n-1}{\ve}\Big)$
yields that
\begin{equation*}
\Big(\int_\Omega |u(x) -b|^{\frac{\delta}{n-1}} d \Ha^{\delta}_\infty \Big)^{\frac{n-1 }{\delta}}
\le  2
\Big(\int_\Omega |u(x) -b|^{\frac{\delta}{n-1-\ve}} d \Ha^{\delta}_\infty \Big)^{\frac{n-1-\ve }{\delta}}
\Big(\Ha^{\delta}_\infty (\Omega)\Big)^{\frac{\ve}{\delta}}\,,
\end{equation*}
Combining this observation to Corollary \ref{epsilon} yields the following corollary
for $0<\delta <n$.
The case $\delta=n$ follows from the classical case by the H\"older inequality.

\begin{corollary}\label{epsilon_righthandside}
Suppose that  $\Omega$ is a bounded  $(\alpha, \beta)$-John domain in $\Rn$, $n\geq 2$.
If $\delta \in (0, n]$
then  for any fixed $\ve \in (0, n-1)$ there exists a constant $c$ such that  the inequality
\begin{equation}\label{equ:epsilon_righthandside}
\inf_{b \in \R}  \Big(\int_\Omega |u(x) -b|^{\frac{\delta}{n-1}} d \Ha^{\delta}_\infty \Big)^{\frac{n-1}{\delta}}
\le c
\Big(\int_{\Omega} |\nabla u(x)|^{\frac{\delta}{n-\ve}} \, d \Ha^{\delta}_\infty \Big)^{\frac{n-\ve}{\delta}}
\end{equation}
holds for all $u\in C^1(\Omega )$; the constant $c$ is independent of $u$.
\end{corollary}

Note that if $p=\delta /n$ then $\delta p /(\delta -p) = \delta/(n-1)$.
Even if it were possible to apply some convergence result to inequality \eqref{equ:epsilon_righthandside}, such as a dominated convergence theorem
from \cite[Proposition 2.4]{PonceSpector23} that  could imply
\begin{equation*}
\int_{\Omega} |\nabla u(x)|^{\frac{\delta}{n-\ve}} \, d \Ha^{\delta}_\infty  \to \int_{\Omega} |\nabla u(x)|^{\frac{\delta}{n}} \, d \Ha^{\delta}_\infty
\end{equation*}
as $\ve\to 0^+$,
it seems that the constant 
$c$  in \eqref{equ:epsilon_righthandside} blows up  when $\ve\to 0^+$.

We show the following inequality where the dimension of the Hausdorff content is smaller on the right-hand side than on the left-hand side.

\begin{theorem}\label{cor:limit-case-1}
Suppose that  $\Omega$ is a bounded  $(\alpha, \beta)$-John domain in $\Rn$, $n\geq 2$.
Let  $\delta \in (0, n]$   and $\ve \in(0, \delta(1-\frac1n))$. Then
the inequality
\begin{equation*}
\inf_{b \in \R}  \Big(\int_\Omega |u(x) -b|^{\frac{\delta}{n-1-\ve n/\delta}} d \Ha^{\delta}_\infty \Big)^
{\frac{\delta(n-1)-\ve n}{\delta^2 }} \le c \Big(\int_{\Omega} |\nabla u(x)|^{\frac{\delta}{n}} \, d \Ha^{\delta-\ve}_\infty \Big)^{\frac{n}{\delta}}
\end{equation*}
holds for some  positive constant $c=c(\alpha, \beta, \delta, \epsilon, n) <\infty$ 
for all $u \in C^1(\Omega)$.

\end{theorem}

\begin{proof}
Let  $\delta\in (0,n]$   and $\ve \in(0, \delta(1-\tfrac1n))$ be fixed.  By Proposition~\ref{GeneralizationOV} with
$\delta_1=\delta-\ve < \delta =\delta_2$
we obtain
\begin{equation*}
\Big(\int_\Omega |u(x) -b|^{\frac{\delta^2 }{\delta(n-1)-\ve n}} d \Ha^{\delta}_\infty \Big)^{\frac{\delta(n-1)-\ve n}{\delta ^2} }
\le  c
  \Big(\int_\Omega |u(x) -b|^{\frac{{\delta}/{n}(\delta- \ve)}{\delta-\ve - {\delta}/{n}}} d \Ha^{\delta -\ve}_\infty \Big)^{\frac{\delta-\ve - {\delta}/{n}}{{\delta}/{n}(\delta- \ve)}},
 \end{equation*}
 where $c$ depends only on $n$ and $\delta$.
 Then we use Theorem~\ref{thm:Poincare} with $\kappa =0$ for the $\delta -\epsilon$-dimensional Hausdorff dimension.
 The assumption $\ve < \delta(1-\frac1n)$ implies that
$\frac{\delta -\ve}{n} < \frac{\delta}{n} < \delta -\ve$.
Thus by Theorem~\ref{thm:Poincare}  
 \begin{equation*}
  \inf_{b \in \R}  \Big(\int_\Omega |u(x) -b|^{\frac{{\delta}/{n}(\delta- \ve)}{\delta-\ve - {\delta}/{n}}} d \Ha^{\delta -\ve}_\infty \Big)^{\frac{\delta-\ve - {\delta}/{n}}{{\delta}/{n}(\delta- \ve)}}
 \le c_2(\alpha, \beta, \delta, \ve, n) \Big(\int_{\Omega} |\nabla u(x)|^{\frac{\delta}{n}} \, d \Ha^{\delta-\ve}_\infty \Big)^{\frac{n}{\delta}}.
\end{equation*}
Combining these estimates yields the claim.
\end{proof}

The previous theorem implies an interesting result, Theorem~\ref{cor:SP-2d-epsilon}.

\begin{proof}[Proof of Theorem~\ref{cor:SP-2d-epsilon}]
Let $\delta \in (0,n]$  be given and then take $\ve \in(0, \delta(1-\frac1n))$.
Now $\frac{\delta (n-1)}{\delta (n-1)-\ve n} > 1$.
Let $b\in\R$ be fixed.
By the Hölder inequality  with 
$\big(\frac{\delta (n-1)}{\epsilon n}\,,\frac{\delta (n-1)}{\delta (n-1)-\epsilon n}\big)$
and the previous theorem we obtain
\[
\begin{split}
& \Big(\int_\Omega |u(x) -b|^{\frac{\delta}{n-1}} d \Ha^{\delta}_\infty \Big)^{\frac{n-1}{\delta}}\\
&\quad \le \bigg(2 \Ha^{\delta}_\infty(\Omega)^{\frac{\ve n}{\delta (n-1)}} \Big(\int_\Omega |u(x) -b|^{\frac{\delta}{n-1} \frac{\delta(n-1)}{\delta(n-1)-\ve n}} d \Ha^{\delta}_\infty \Big)^{\frac{\delta(n-1)-\ve n}{\delta(n-1)}} \bigg)^{\frac{n-1}{\delta}}\\
&\quad \le c_1  \Big( \int_\Omega |u(x) -b|^{\frac{\delta}{n-1-\ve n/\delta}} d \Ha^{\delta}_\infty  \Big)^{\frac{n-1-\ve n/\delta}{\delta}}\\
&\quad \le c_2 \Big(\int_{\Omega} |\nabla u(x)|^{\frac{\delta}{n}} \, d \Ha^{\delta- \ve}_\infty \Big)^{{\frac{n}{\delta}}}
\end{split}
\]
for all $\ve \in(0, \delta(1-\frac1n))$ and all $u \in C^1(\Omega)$.
Here the  positive constants $c_1=c_1(\delta, \ve, n, \Ha^{\delta}_\infty(\Omega) )$ and
$c_2=c_2(\alpha, \beta, \delta, \ve, n, \Ha^{\delta}_\infty(\Omega) )$ are finite.
Note that  the assumption $\ve \in(0, \delta(1-\frac1n))$ means that $\frac{\ve n}{\delta (n-1)}<1$, and hence  in fact $\Ha^{\delta}_\infty(\Omega)^{\frac{\ve n}{\delta (n-1)}} \le \max\{1, \Ha^{\delta}_\infty(\Omega)\}$.
\end{proof}

Without using Theorem \ref{cor:limit-case-1} we have the following result
 which differs slightly from the inequality in Theorem~\ref{cor:SP-2d-epsilon}.                  
Here the relation of the  Hausdorff dimension and the exponent of the function is more flexible.

\begin{theorem}\label{deltas}
Suppose that  $\Omega$ is a bounded  $(\alpha, \beta)$-John domain in $\Rn$, $n\geq 2$.
If $0<\delta \le n-1$ and $0<\delta_2<n\delta/(n-1)$ or $n-1\le \delta \le n$ and $0<\delta_2\le n$,
then the inequality
\begin{equation*}
\inf_{b \in \R}  \Big(\int_\Omega |u(x) -b|^{\frac{\delta}{n-1}} d \Ha^{\delta_2}_\infty \Big)^{\frac{n-1}{\delta}}
\le c
\Big(\int_{\Omega} |\nabla u(x)|^{\frac{\delta}{n}} \, d \Ha^{\delta_2(1-\frac{1}{n})}_\infty \Big)^{\frac{n}{\delta}}
\end{equation*}
holds for all $u\in C^1(\Omega )$.
\end{theorem}

\begin{proof}
Let $b\in R$ and $\delta \in (0,n]$ be given.
Let us take two variables $\delta_1$ and $\delta_2$ which we fix later.
Whenever $ 0<\delta_1 <\delta_2\le n$, then by Proposition \ref{GeneralizationOV} 
\begin{equation*}
  \Big(\int_\Omega |u(x) -b|^{\frac{\delta}{n-1}} d \Ha^{\delta_2}_\infty \Big)^{\frac{n-1}{\delta}}
\le  \biggl(\frac{\delta_2}{\delta_1}\biggr)^{\frac{n-1}{\delta}}
  \Big(\int_\Omega |u(x) -b|^{\frac{\delta\delta_1}{(n-1)\delta_2}} d \Ha^{\delta_1}_\infty \Big)^{\frac{\delta_2(n-1)}{\delta\delta_1}}.
\end{equation*}
Whenever 
\begin{equation*}
\frac{\delta\delta_1}{(n-1)\delta_2}>\frac{\delta_1}{n}, \mbox{ that is }
\delta >(1-\frac{1}{n})\delta_2\,,
\end{equation*}
the Poincar\'e inequality \cite[Theorem 3.2]{HH-S_JFA}
yields that
\begin{equation*}
 \inf_{b \in \R}  \Big(\int_\Omega |u(x) -b|^{\frac{\delta\delta_1}{(n-1)\delta_2}} d \Ha^{\delta_1}_\infty \Big)^{\frac{\delta_2(n-1)}{\delta\delta_1}}
 \le c_1 \Big(\int_{\Omega} |\nabla u(x)|^{\frac{\delta\delta_1}{(n-1)\delta_2}} \, d \Ha^{\delta_1}_\infty \Big)^{\frac{(n-1)\delta_2}{\delta\delta_1}}
\end{equation*}
with some constant $c_1=c_1(\alpha, \beta, \delta, \delta_1, \delta_2, n)< \infty$ independent of $u$.
Whenever
\begin{equation*}
\frac{\delta\delta_1}{(n-1)\delta_2}=\frac{\delta}{n}, \mbox{ that is }
\delta_1 =(1-\frac{1}{n})\delta_2\,,
\end{equation*}
then  the equality
\begin{equation*}
\Big(\int_{\Omega} |\nabla u(x)|^{\frac{\delta\delta_1}{(n-1)\delta_2}} \, d \Ha^{\delta_1}_\infty \Big)^{\frac{(n-1)\delta_2}{\delta\delta_1}}
=
\Big(\int_{\Omega} |\nabla u(x)|^{\frac{\delta}{n}} \, d \Ha^{\delta_1}_\infty \Big)^{\frac{n}{\delta}}
\end{equation*}
is valid for $u\in C^1(\Omega)$.
Hence  the inequality 
\begin{equation*}
\inf_{b \in \R}  \Big(\int_\Omega |u(x) -b|^{\frac{\delta}{n-1}} d \Ha^{\delta_2}_\infty \Big)^{\frac{n-1}{\delta}}
\le c_2
\Big(\int_{\Omega} |\nabla u(x)|^{\frac{\delta}{n}} \, d \Ha^{\delta_1}_\infty \Big)^{\frac{n}{\delta}}
\end{equation*}
follows whenever
\[\delta_1=\delta_2(1-\frac{1}{n}) \quad
\text{and} \quad \delta >\delta_2(1-\frac{1}{n}).
\]
Thus,
$\delta_2<n\delta /(n-1)$, and here $n\delta/(n-1)$ could be larger than $n$ or less than or equal to  $n$. This causes
two parts for the assumption.
Here $c_2=c_2(\alpha, \beta, \delta, \delta_1, \delta_2, n)< \infty$.
\end{proof}

If we require  from  functions only $\frac{\delta}{n}$-integrability 
of the absolute value of the gradient with respect to the \ $\Ha_\infty^{\delta}$-content, then we have the following weak-type Poincar\'e-Sobolev inequality.

\begin{theorem}\label{weaklimit_kappa}
Suppose that  $\Omega$ is a bounded  $(\alpha, \beta)$-John domain in $\Rn$.
Let  $\delta \in (0, n]$  and $\kappa \in [0,1)$ be given. 
If $u \in C^1(\Omega)$, then for every $t>0$
\begin{equation*}
\inf_{b \in \R} \, \Ha_\infty^{\delta\frac{n-\kappa}{n}}\biggl(\{x\in\Omega :\vert u(x)-b\vert >t\}\biggr)^{\frac{n-1}{\delta}}
\le \frac{c}{t} \Big(\int_{\Omega}\vert \nabla u(x)\vert^{\frac{\delta}{n}}\,d\Ha_\infty^{\delta} \Big)^{\frac{n}{\delta}}\,,
\end{equation*}
where
$c$ is a constant which depends only on 
$n$, $\delta$, $\kappa$,
and John constants $\alpha$ and $\beta$.  
\end{theorem}

\begin{proof}
Since $\Omega$ is an  $(\alpha, \beta)$-John domain, there exists
a pointwise estimate
\begin{equation}\label{pointwise}
|u(x) -u_B| \le c(n, \alpha, \beta) \int_\Omega \frac{|\nabla u(y)|}{|x-y|^{n-1}}\,dy =  c(n, \alpha, \beta )I_1 |\nabla u| (x)
\end{equation}
for every $x \in \Omega$.   Here, $B=B(x_0, c(n)\alpha^2/\beta)$ and
$c(n,\alpha ,\beta)=c(n)(\beta /\alpha)^{2n}$.
We refer to the history of  (\ref{pointwise}) in  the proof of \cite[Theorem 3.2]{HH-S_JFA}.

We use \cite[Lemma 2]{HH-S_Proc2022}  with $s=1$ to estimate the Riesz potential by the fractional maximal operator, and obtain 
\begin{equation}\label{pointwisemax}
|u(x) -u_B| \le c (M_\kappa |\nabla u|(x))^{\frac{n-1}{n- \kappa}} N^{1-\frac{n-1}{n-\kappa}},
\end{equation}
where $N:= \Big(\int_{\Rn\setminus B(x, r)} |\nabla u(y)|^{\frac{\delta}{n}} \, d \Ha^{\delta}_\infty \Big)^{\frac{n}{\delta}}$.
Here  the assumption $\kappa \in[0, 1)$ comes in.
Combining estimates \eqref{pointwise} and \eqref{pointwisemax} yields that
\[
\begin{split}
&\Ha_\infty^{\delta\frac{n-\kappa}{n}}\Big(\big\{x\in\Omega :\vert u(x)-u_B\vert >t \big\}\Big)\\
&\quad\le  
\Ha_\infty^{\delta\frac{n-\kappa}{n}}
\biggl(\Big\{x\in\Omega :c (M_{\kappa}|\nabla u|(x))^{\frac{n-1}{n-\kappa}} N^{\frac{n-\kappa +1-n}{n-\kappa}} >t \Big\}\biggr)\\
&\quad=  
\Ha_\infty^{\delta\frac{n-\kappa}{n}}
\biggl(\Big\{x\in\Omega :  M_{\kappa} \big(c N^{\frac{1-\kappa}{n-1}}|\nabla u|\big)(x) >t^{\frac{n-\kappa}{n-1}}\Big\}\biggr)\,.
\end{split}
\]
Using the weak type estimate for the fractional maximal operator proved by Adams \cite[Theorem 7 (ii)]{Adams98} implies that
\[
\begin{split}
&\Ha_\infty^{\delta\frac{n-\kappa}{n}}
\Big(\big\{x\in\Omega :\vert u(x)-u_B\vert >t \big\}\Big)\\
&\quad\le c(t^{\frac{n-\kappa}{n-1}})^{-\frac{\delta}{n-\kappa}} \bigg(\int_{\Omega}\big(c N^{\frac{1-\kappa}{n-1}}\vert \nabla u(x)\vert \big)^{\frac{\delta}n}\,d\Ha_\infty^{\delta} \bigg)^{\frac{n}{\delta}\frac{\delta}{n-\kappa}}\\
&\quad\le c 
t^{-\frac{\delta}{n-1}}
 \Big(\int_{\Omega}\vert \nabla u(x)\vert^{\frac{\delta}n}\,d\Ha_\infty^{\delta} \Big)^{\frac{n}{n-1}}\,.
\end{split}
\]
This gives the claim.
\end{proof}

The previous theorem yields the following corollary when we choose $\kappa =0$.

\begin{corollary}\label{weaklimit}
Suppose that  $\Omega$ is a bounded  $(\alpha, \beta)$-John domain in $\Rn$ and
$\delta \in (0, n]$  is given. If
$u \in C^1(\Omega)$, then for every $t>0$
\begin{equation*}
\inf_{b \in \R} \, \Ha_\infty^{\delta}\biggl(\{x\in\Omega :\vert u(x)-b\vert >t\}\biggr)^{\frac{n-1}{\delta}}
\le \frac{c}{t} \Big(\int_{\Omega}\vert \nabla u(x)\vert^{\frac{\delta}{n}}\,d\Ha_\infty^{\delta} \Big)^{\frac{n}{\delta}}\,,
\end{equation*}
where $c$ is a constant which depends only on $n$, $\delta$, 
and John constants $\alpha$ and $\beta$.
\end{corollary}

\begin{remark}
We point out that  the proofs of 
the theorems in this section, including  the proof for Theorem~\ref{cor:SP-2d-epsilon},  
give stronger inequalities than in the statements in these theorems, respectively. 
Namely,  in the proofs we estimate $\vert u(x)-u_B\vert $ where
$B:=B(x_0,k\dist (x_0,\partial \Omega ))$ 
and $u_B=\vert B\vert^{-1}\int_B u(x)\,dx$
with $k\in(0,1)$  being a constant and depending 
on $x_0$ and John constants of $\Omega$.
Hence, for example the inequality in  Theorem~\ref{cor:SP-2d-epsilon} can be replaced by
the inequality
\begin{equation*}
 \Big(\int_\Omega |u(x) -u_B|^{\frac{\delta}{n-1}} d \Ha^{\delta}_\infty \Big)^
{\frac{n-1}{\delta}} \le c \Big(\int_{\Omega} |\nabla u(x)|^{\frac{\delta}{n}} \, d \Ha^{\delta- \ve}_\infty \Big)^{\frac{n}{\delta}}\,,
\end{equation*}
holding for all $\ve \in(0, \delta(1-\frac1n))$ and
for all $u \in C^1(\Omega)$.  Here $c$ is a finite constant independent of $u$.
\end{remark}

\section{Inequalities  in  the limiting  case $p= \delta$}\label{Section:Trudinger}

The main result  in this section is the following theorem  which  improves 
\cite[Corollary 1.2, Theorem 5.9]{HH-S_AAG} 
whenever the set where the functions are defined is a  bounded John domain.

\begin{theorem}\label{thm:main-p=n}
If $\Omega$ is a bounded John domain in $\R^n$, $n\geq 2$, and  $\delta_1, \delta_2\in (0, n]$, 
then there exist positive constants $a$ and $b$  independent of functions 
$u \in C^{1}(\Omega)$
such that
 \begin{equation}\label{basicTr}
\int_{\Omega} \exp\big( a |u(x)-u_B|^{\frac{n}{n-1}} \big) \, d \Ha^{\delta_1}_\infty \le b
\end{equation}
 whenever $\int_{\Omega} |\nabla u(y)|^{\delta_2} \, d \Ha^{\delta_2}_\infty \le 1$.
 Here
 $B:= B(x_0, k \dist(x_0, \partial \Omega ))$ with $k\in (0,1)$  
 depending  on $\Omega$ and $x_0\in \Omega$ is  fixed.
\end{theorem}

We 
provide a proof
for the convenience of the reader although it follows the same lines as the
proof for \cite[Theorem 5.9]{HH-S_AAG}  where  functions from the variable order Sobolev spaces defined  on more irregular domains
 were considered.
The proof of Theorem \ref{thm:main-p=n}  is based on the following result.

\begin{theorem}\label{thm:Riesz-limit-case}
Let $\Omega$ be an open and  bounded set in $\Rn$. Let
 $\delta_1, \delta_2 \in  (0, n]$ and $\eta\in (0, n)$.
Then, there exist constants $c_1$ and $c_2$  independent of the function $f$ such that the inequality
\begin{equation*}
\Ha_\infty^{\delta_1}(\{x \in \Omega: I_{\eta}f(x) >t\})
\le c_1\exp (-c_2t^{\frac{n}{n-\eta}})
\end{equation*}
holds
for all
$f \in L^{1}_{\loc}(\Omega)$ with $\int_{\Omega} |f(y)|^{ \frac{\delta_2}{\eta}} \, d \Ha^{\delta_2}_\infty \le 1$.
\end{theorem}

We refer  to  \cite[Theorem 1.3]{MartinezSpector21}    which has
slightly different assumptions. Note that if $f \in L^1_{\loc}(\Omega)$ and $\int_{\Omega} |f(y)|^{ \frac{\delta_2}{\eta}} \, d \Ha^{{\delta_2}}_\infty < \infty$, then   
$\|f\|_{L^{n/\eta}(\Omega)} \le c$  with some finite constant $c$ by Proposition~\ref{GeneralizationOV}. 
Hence,  it is possible to use the proof of \cite[Theorem 1.3]{MartinezSpector21} with a slight modification.
Indeed, the proof for Theorem~\ref{thm:Riesz-limit-case} relies on 
Lemma \ref{lem:Hedberg} and Lemma \ref{lem:Haudorff-1}  
which we recall next. 
 
\begin{lemma}[Hedberg-type estimate]\label{lem:Hedberg}
Let $\Omega$ in $\Rn$ be a bounded, open set.  
Let  $\delta \in (0, n]$, $\eta\in (0, n)$, and $p \in (1, \delta/\eta)$.
Then 
there exists a constant $c$ such that
for every $\ve \in(0, \eta]$ for all $x \in \Omega$, the inequality
\[
I_{\eta}f(x) \le c\, \max\Big\{\frac{1}{ 2^{\ve}-1}, \Big(\frac{1}{\delta - \eta p} \Big )^{\frac{np-\delta}{np}}    \Big\} (\M_{\eta-\ve} f(x))^{\frac{\delta -\eta p}{\delta- \eta p +\ve p}}
\]
holds
for all $f \in L^{1}_{\loc}(\Omega)$ with $\int_{\Omega} |f(y)|^{p} \, d \Ha^{\delta}_\infty  \le 1$. 
\end{lemma}

\begin{lemma} \label{lem:Haudorff-1}
Let $\Omega$ be a  bounded, open set in $\Rn$.
If $\eta \in [0, n)$,
then there exists a constant $c$, depending only on the dimension $n$,  such that the inequality
\[
\Ha_\infty^{n- \eta}(\{x \in \Omega: \M_{\eta}f(x) >t\})
\le  \frac{c(n)}{t} \|f\|_{L^1(\Omega)}
\]
holds for all $f \in L^1(\Omega)$.
\end{lemma}

A  proof for Lemma \ref{lem:Hedberg}  is similar to the proof of  \cite[Lemma 3.4]{MartinezSpector21} where functions from Lorentz spaces are considered.
For the proof  of Lemma \ref{lem:Haudorff-1} we refer to
\cite[Lemma 3.5]{MartinezSpector21} and also to \cite[Lemma 4.6]{HH-S_AAG} where a variable order Riesz potential version is  considered.
Now we are ready for the proof for Theorem \ref{thm:main-p=n}. 

 \begin{proof}[Proof of Theorem \ref{thm:main-p=n}]
 Let $B=B(x_0, c_0(\Omega)\dist (x_0,\partial\Omega ))$ be a ball in $\Omega$  appearing in \eqref{pointwise}.
Let $m >0$ be a constant that will be chosen later.  
The pointwise estimate (\ref{pointwise}) yields that there exists a positive constant $c_1<\infty$ independent of $u$ such that
 \[
 \begin{split}
 &\int_{\Omega} \exp\bigg(\bigg (\frac{m}{1+|\Omega|} \vert u(x) - u_B\vert  \bigg)^{\frac{n}{n-1}}\bigg) \, d \Ha^{\delta_1}_\infty \\
 &\quad\le   \int_{\Omega} \exp\bigg( \bigg(\frac{m c_1}{1+|\Omega\vert}  I_{1} |\nabla u| (x)\bigg)^{\frac{n}{n-1}} \bigg) \, d \Ha^{\delta_1}_\infty \,.
 \end{split}
 \]
 When we use the definition of the Choquet integral and split the interval of integration into two parts we obtain
 \[
 \begin{split}
 &\int_{\Omega} \exp\bigg(\bigg (\frac{m}{1+|\Omega|} \vert u(x) - u_B\vert  \bigg)^{\frac{n}{n-1}}\bigg) \, d \Ha^{\delta_1}_\infty \\
 &\quad  = \int_0^\infty \Ha^{\delta_1}_\infty\Bigg(\bigg\{x \in \Omega: \exp\bigg( \bigg(\frac{m c_1}{1+|\Omega|}  I_{1} |\nabla u|(x) \bigg)^{\frac{n}{n-1}}\Bigg) >t \bigg\} \Bigg) \, dt\\
 &\quad= \int_0^1 \Ha^{\delta_1}_\infty\Bigg(\bigg\{x \in\Omega : \bigg(\frac{ I_{1} |\nabla u|(x)}{1+|\Omega|} \bigg)^{\frac{n}{n-1}}> \frac{\log (t)}{(mc_1)^{n/(n-1))}}\bigg\}\Bigg) \, dt\\
&\qquad +  \int_1^\infty \Ha^{\delta_1}_\infty \Bigg(\bigg\{x \in \Omega : \bigg(\frac{ I_{1} |\nabla u|(x)}{1+|\Omega|}\bigg)^{\frac{n}{n-1}} > \frac{\log (t)}{(mc_1)^{{n}/(n-1)}}\bigg\}\Bigg)\, dt\,.
\end{split}
\]
The integral over the  unit interval is controlled by
$ \Ha^{\delta_1}_\infty(\Omega)$.
For the second integral over the unbounded interval 
we apply  Theorem~\ref{thm:Riesz-limit-case} to  $|\nabla u|/(1+ |D|)$ 
in the case $\eta =1$ and obtain
\[
\begin{split}
 &\Ha^{\delta_1}_\infty \Bigg(\bigg\{x \in \Omega : I_{1} \bigg(\frac{ |\nabla u|}{1+|\Omega|}\bigg)(x) > \bigg(\frac{\log (t)}{(mc_1)^{n/(n-1)}} \bigg)^{\frac{n-1}{n}}\bigg\}\Bigg)\\
  &\qquad\le 
 c_2 \exp\bigg(- c_3 \frac{\log (t)}{(mc_1)^{n/((n-1))}} \bigg) 
 = c_2 t^{-\frac{c_3}{(mc_1)^{n/(n-1)}}}.
\end{split}
\] 
Choosing $m>0$ to be so small that $\frac{c_3}{(mc_1)^{n/(n-1)}}>1$ implies
\[
 \begin{split}
 &\int_1^\infty \Ha^{\delta_1}_\infty \Bigg(\bigg\{x \in \Omega :  I_{1} \bigg(\frac{ |\nabla u|}{1+|\Omega\vert}\bigg)(x) > \bigg(\frac{\log (t)}{(mc_1)^{n/(n-1)}} \bigg)^{\frac{n-1}{n}}\bigg\}\Bigg)\, dt\\
 &\qquad \le \int_1^\infty c_2 t^{-\frac{c_3}{(mc_1)^{n/(n-1)}}} \, dt =: c_4< \infty.
 \end{split}
 \]
Inequality (\ref{basicTr}) follows by taking
\[
a=\biggl(\frac{m}{1+\vert \Omega\vert}\biggr)^{\frac{n}{n-1}} \mbox {and }
b= \Ha^{\delta_1}_\infty(\Omega) + c_4\,. \qedhere
\]
 \end{proof}

\section{ Results for $C^1_0$-functions in the case $\frac{\delta}{n} \le p \le \delta$}\label{Section:C^1_0}

In this section we give  to $C^1_0$-functions
some of the  corresponding results of the previous  sections, Sections \ref{Section:PS} and \ref{Section:Trudinger}.
We recall the following theorem from \cite{HH-S_JFA}.

\begin{theorem}\cite[Theorem 4.2]{HH-S_JFA}\label{thm:Poincare-0}
Let $\Omega $ is  a bounded domain  in $\Rn$, $\delta  \in (0, n]$, and   $p \in (\delta/n, \delta)$.
If  $\kappa \in [0, 1)$, then  there exists a constant $c$ depending only on $n$, $\delta$,
$\kappa$, and $p$ such that  
\[
\Big(\int_\Omega |u(x)|^{\frac{p(\delta- \kappa p) }{\delta-p}} d \Ha^{\delta -\kappa p}_\infty \Big)^{\frac{\delta-p}{p(\delta- \kappa p)}}
\le c \Big(\int_{\Omega} |\nabla u(x)|^{p} \, d \Ha^{\delta}_\infty \Big)^{\frac{1}{p}}
\] 
for all $u \in C^1_0(\Omega)$.
\end{theorem}

This theorem implies the following corollaries.
When we choose  $\kappa =0$ and the exponent of the right-hand side 
of the absolute value of the gradient
to be $\delta/(n-\ve )$, Corollary \ref{epsilon_0} follows.
Using H\"older's inequality and Corollary \ref{epsilon_0} yields Corollary \ref{epsilon_righthandside}.

\begin{corollary}\label{epsilon_0}
Suppose that  $\Omega$ is a bounded  domain in $\Rn$, $n\geq 2$.
If $\delta \in(0, n] $,
then  for any fixed $\ve \in (0, n-1)$ there exists a constant $c$ such that  the inequality
\begin{equation*}
\Big(\int_\Omega |u(x) |^{\frac{\delta}{n-1-\ve}} d \Ha^{\delta}_\infty \Big)^{\frac{n-1-\ve }{\delta}}
\le c
\Big(\int_{\Omega} |\nabla u(x)|^{\frac{\delta}{n-\ve}} \, d \Ha^{\delta}_\infty \Big)^{\frac{n-\ve}{\delta}}
\end{equation*}
holds for all $u\in C^1_0(\Omega )$,
the constant $c$ being independent of $u$.
\end{corollary}

\begin{corollary}\label{epsilon_righthandside}
Suppose that  $\Omega$ is a bounded  domain in $\Rn$, $n\geq 2$.
If $\delta \in (0, n]$,
then  for any fixed $\ve \in (0, n-1)$ there exists a constant $c$ such that  the inequality
\begin{equation*}
\Big(\int_\Omega |u(x) |^{\frac{\delta}{n-1}} d \Ha^{\delta}_\infty \Big)^{\frac{n-1}{\delta}}
\le c
\Big(\int_{\Omega} |\nabla u(x)|^{\frac{\delta}{n-\ve}} \, d \Ha^{\delta}_\infty \Big)^{\frac{n-\ve}{\delta}}
\end{equation*}
holds for all $u\in C^1_0(\Omega )$, the constant $c$ being independent of $u$.
\end{corollary}

If we use Theorem \ref{thm:Poincare-0} instead of 
Theorem \ref{thm:Poincare} in the proof of
Theorem \ref{cor:limit-case-1} and work through the proof of Theorem  \ref{cor:SP-2d-epsilon},
we obtain the following result.

\begin{theorem}
Suppose that  $\Omega$ is a bounded domain in $\Rn$,  $n\geq 2$.
If $\delta \in (0,n]$ is fixed, then the inequality
\begin{equation*}
  \Big(\int_\Omega |u(x)|^{\frac{\delta}{n-1}} d \Ha^{\delta}_\infty \Big)^
{\frac{n-1}{\delta}} \le c \Big(\int_{\Omega} |\nabla u(x)|^{\frac{\delta}{n}} \, d \Ha^{\delta- \ve}_\infty \Big)^{\frac{n}{\delta}}
\end{equation*}
holds for all $\ve \in(0, \delta(1-\frac1n))$ and
for all $u \in C^1_0(\Omega)$.
 Here $c$ is a constant depending only on $\delta, \ve, n,$ and
$\Ha^{\delta}_\infty(\Omega)$.
\end{theorem}

Following the proof for Theorem \ref{deltas},
albeit refering to \cite[Theorem 4.7 (a)]{HH-S_JFA}
instead of \cite[Thereom 3.2]{HH-S_JFA},
yields a result  with  more flexibility in  the  Hausdorff dimensions and exponents for $C^1_0$-functions, too.
\begin{theorem}
Suppose that  $\Omega$ is a bounded  domain in $\Rn$, $n\geq 2$.
If $0<\delta <n-1$ and
$0<\delta_2<\frac{n\delta}{n-1}$,
or 
if $n-1\le \delta\le n$ and $0<\delta_2 \le n$,
then the inequality
\begin{equation*}
\Big(\int_\Omega |u(x)|^{\frac{\delta}{n-1}} d \Ha^{\delta_2}_\infty \Big)^{\frac{n-1}{\delta}}
\le c
\Big(\int_{\Omega} |\nabla u(x)|^{\frac{\delta}{n}} \, d \Ha^{\delta_2(1-\frac{1}{n})}_\infty \Big)^{\frac{n}{\delta}}
\end{equation*}
holds for all $u\in C^1_0(\Omega )$.
\end{theorem}

Let us recall 
the pointwise estimate
\begin{equation}\label{pointwise_0}
|u(x)| \le c(n) \int_\Omega \frac{|\nabla u(y)|}{|x-y|^{n-1}}\,dy
\end{equation}
that holds for  $u \in C^1_0(\Omega)$ and all $x \in \Omega$ 
\cite[Remark 2.8.6]{Ziemer}.

By working as in the proof of Theorem~\ref{weaklimit_kappa} we obtain the following theorem.
Indeed, in the proof we just use
(\ref{pointwise_0})
instead of (\ref{pointwise}).

\begin{theorem}\label{weaklimit-0-kappa}
Suppose that  $\Omega$ is a bounded domain in $\Rn$. Let 
$\delta \in (0, n]$  and $\kappa\in [0,1)$ be given. If 
$u \in C^1_0(\Omega)$,
then  for every $t>0$
\begin{equation*}
\Ha_\infty^{\delta\frac{n-\kappa}{n}}\biggl(\{x\in\Omega :\vert u(x)\vert >t\}\biggr)^{\frac{n-1}{\delta}}
\le \frac{c}{t} \Big(\int_{\Omega}\vert \nabla u(x)\vert^{\frac{\delta}{n}}\,d\Ha_\infty^{\delta} \Big)^{\frac{n}{\delta}}\,,
\end{equation*}
where $c$ is a constant which depends only on $n$, $\delta$, and $\kappa$.
\end{theorem}

Choosing $\kappa =0$ in Theorem
\ref{weaklimit-0-kappa}  implies Corollary \ref{weaklimit-0-0}.

Combining ({\ref{pointwise_0}) and \cite[Corollary 1.4]{MartinezSpector21} yields the following theorem.
We point out that 
\cite[Corollary 1.4]{MartinezSpector21}
can be used since $  \int_{\Omega} |\nabla u(y)|^{\delta} \, d \Ha^{\delta}_\infty \le 1$ gives that $\|\nabla u\|_{L^n(\Omega)} \le c$
 with some constant $c$  which is independent of $u$ by Proposition~\ref{GeneralizationOV}.

\begin{theorem} 
 Let $\Omega$ be a bounded domain in $\R^n$, $n\geq 2$, and let $\delta_1, \delta_2 \in (0, n]$.  
  Then there exist positive constants $a$ and $b>0$ such that
 \[
\int_{\Omega} \exp\big( a |u(x)|^{\frac{n}{n-1}} \big) \, d \Ha^{\delta_1}_\infty \le b
\]
 for all $u \in C^{1}_0 (\Omega )$ with
$  \int_{\Omega} |\nabla u(y)|^{\delta_2} \, d \Ha^{\delta_2}_\infty \le 1$.
\end{theorem}


\section{Inequalities in the case $p>\delta$}\label{Section:p>delta}

Let us first prove a H\"older continuity result for $C^1$-functions with a finite Choquet integral of the absolute value  
of the gradient with respect to the $\delta$-dimensional  Hausdoff content.

\begin{theorem}\label{Morrey_first}
Let $\delta \in(0, n]$ and  $p\in (\delta,\infty)$.  If $u\in C^1(\R^n)$, then 
\begin{equation}\label{Morrey1}
\sup_{x\neq y \in \Rn}
\frac{\vert u(x)-u(y)\vert }{\vert x-y\vert ^{1-\frac{\delta}{ p}}}
\le c
\Big(
\int_{\Rn} |\nabla u(y)|^{p} \, d \Ha^{\delta}_\infty 
\Big)^{\frac{1}{p}}\,,
\end{equation}
where $c$ is  a constant which  depends only on $\delta$, $n$,  and $p$.
\end{theorem}

\begin{proof}
We may assume that the right-hand side of  (\ref{Morrey1})  is finite. We follow  some parts of \cite[5.6.2]{E} in the proof.
Let us pick up two points $x$ and $y$ in $\Rn$ and write $r=\vert x- y \vert$.
Then by the triangle inequality
\begin{align*}
&\vert u(x)-u(y)\vert\le
\frac{1}{\vert B(x,r)\cap B(y,r)\vert}
\int_{B(x,r)\cap B(y,r)}
\vert u(x)-u(v)\vert\,dv\\
&+\frac{1}{\vert B(x,r)\cap B(y,r)\vert}
\int_{B(x,r)\cap B(y,r)}
\vert u(y)-u(v)\vert\,dv\\
&\le
\frac{c(n)}{\vert B(x,r)\vert}
\int_{B(x,r)}
\vert u(x)-u(v)\vert\,dv
+
\frac{c(n)}{\vert B(y,r)\vert}
\int_{B(y,r)}
\vert u(y)-u(v)\vert\,dv\,.
\end{align*}
We use the well-known inequality for $C^1$-functions \cite[Lemma 4.1]{EG}: 
\begin{equation*}
\frac{1}{\vert B(x,r)\vert}
\int_{B(x,r)}
\vert u(y)-u(x)\vert\,dy
\le c(n)
\int_{B(x,r)}
\frac{\vert \nabla u(y)\vert}{\vert x -y\vert ^{n-1}}\,dy\,.
\end{equation*}
Hence,  using  the H\"older inequality with $\Big(\frac{np}{\delta},\frac{np}{np-\delta}\Big)$ and
Proposition \ref{GeneralizationOV}
gives
\begin{align*}
&\int_{B(x,r)}
\frac{\vert \nabla u(y)\vert}{\vert x -y\vert ^{n-1}}\,dy
\le
\Big(\int_{B(x,r)}\vert\nabla u(y)\vert^{\frac{np}{\delta}}\,dy\Big)^{\frac{\delta}{np}}
\Big(\int_{B(x,r)} \vert x-y\vert^{\frac{np(1-n)}{np-\delta}}\,dy\Big)^{\frac{np-\delta}{np}}\\
& \le c(\delta , n , p)
\Big(\int_{B(x,r)} |\nabla u(y)|^{p} \, d \Ha^{\delta}_\infty \Big)^{\frac1p} r^{\frac{p-\delta}{p}}\,,
\end{align*}
where the assumption $p>\delta$ is used when the integral of $\vert x-y\vert^{\frac{np(1-n)}{np-\delta}}$ 
over a ball $B(x,r)$  with respect to $y$ is calculated.
Combining the above estimates yields the claim.
\end{proof}

The following theorem gives an upper bound for a sup-norm of $C^1$-functions in $\Rn$.

\begin{theorem}\label{Morrey_second}
Let $u\in C^1(\Rn )$.  If $\delta \in(0, n]$ and  $p\in (\delta,\infty)$, then
\begin{equation}\label{Morrey2}
\sup_{x\in\Rn}\vert u(x) \vert 
\le c
\bigg(
\Big(
\int_{\Rn} |u(y)|^{p} \, d \Ha^{\delta}_\infty 
\Big)^{\frac{1}{p}}
+
\Big(
\int_{\Rn} |\nabla u(y)|^{p} \, d \Ha^{\delta}_\infty 
\Big)^{\frac{1}{p}}
\biggr)\,,
\end{equation}
where $c$ is a constant which depends only on $\delta$, $n$, and $p$.
\end{theorem}

Although the proof follows  \cite[5.6.2 Theorem 4]{E} of the classical case, we present the proof for the convenience of the reader.

\begin{proof}
We may assume that the integrals on the right-hand side  of (\ref{Morrey2}) are finite.
Let $x\in\Rn$ be fixed. 
By the triangle inequality
\begin{equation*}
\vert u(x)\vert\le
\frac{1}{\vert B(x,1)\vert}
\int_{B(x,1)}
\vert u(x)-u(y)\vert\,dy
+
\frac{1}{\vert B(x,1)\vert}
\int_{B(x,1)}
\vert u(y)\vert dy\,.
\end{equation*}
The first term on the right-hand side can be estimated as in the proof of Theorem \ref{Morrey_first}.
For the second term on the left we use the H\"older inequality with
$\Big(\frac{np}{\delta},\frac{np}{np-\delta}\Big)$ and
Proposition \ref{GeneralizationOV} 
\begin{align*}
&\frac{1}{\vert B(x,1)\vert}
\int_{B(x,1)}
\vert u(y)\vert dy\le
\frac{1}{\vert B(x,1)\vert}
\Big(\int_{B(x,1)}
\vert u(y)\vert^{\frac{np}{\delta}} \,dy\Big)^{\frac{\delta}{np}}
\vert B(x,1)\vert ^{1-\frac{\delta}{np}}\\
&\le c(\delta , n, p)
\Big(
\int_{\Rn} |u(y)|^{p} \, d \Ha^{\delta}_\infty 
\Big)^{\frac{1}{p}}\,.
\end{align*}
Combining the estimates implies the claim.
\end{proof}

Theorem \ref{Morrey_first} and  Theorem \ref{Morrey_second} imply a Morrey inequality for $C^1$ functions in $\Rn$.

\begin{corollary}\label{Morrey_third}
Let $u\in C^1(\Rn )$, $\delta \in(0, n]$, and  $p\in (\delta,\infty)$.
Then there exists a constant  $c$ depending only on $\delta$, $n$, and $p$ such that
\begin{equation*}
\sup_{x\in\Rn}\vert u(x) \vert 
+
\sup_{x\neq y \in\Rn}
\frac{\vert u(x)-u(y)\vert }{\vert x-y\vert ^{1-\frac{\delta}{p}}}
\le c
\bigg(
\Big(
\int_{\Rn} |u(y)|^{p} \, d \Ha^{\delta}_\infty 
\Big)^{\frac{1}{p}}
+
\Big(
\int_{\Rn} |\nabla u(y)|^{p} \, d \Ha^{\delta}_\infty 
\Big)^{\frac{1}{p}}
\biggr).
\end{equation*}
\end{corollary}

Using the well-known Morrey inequality
\cite[5.6.2]{E}  and Proposition \ref{GeneralizationOV} with $0< \delta \le n < q$ implies that
\begin{align*}
\vert u(x)-u(y)\vert 
& \le c_1 r\Big(\frac{1}{\vert B(x,2r)\vert}\int_{B(x,2r)}\vert \nabla u(v)\vert ^q\, d\Ha^{n}_\infty\Big)^{\frac{1}{q}}\\
&\le c_2 r\Big(\frac{1}{\vert B(x,2r)\vert^{\frac{\delta}{n}}}\int_{B(x,2r)}\vert \nabla u(v)\vert ^{\frac{q\delta}{n}}\, d\Ha^{\delta}_\infty\Big)^{\frac{n}{q\delta}}\,.
\end{align*}
Note that here $\frac{q\delta}{n} \in(\delta, \infty)$ whenever $q \in(n, \infty)$.
Writing $p=\frac{q\delta}{n}$ yields the following theorem, which also follows
from the proof of Theorem \ref{Morrey_first}.

\begin{theorem}\label{Morrey_two_parts}  
Suppose that $u\in C^1(B(x,2r) )$.
If 
$\delta \in(0, n]$ and  $p\in (\delta,\infty)$,  then for every
 $y\in B(x,r)$
\begin{equation*}
\vert u(x)-u(y)\vert \le c r\Big(\frac{1}{\vert B(x,2r)\vert^{\frac{\delta}{n}}}\int_{B(x,2r)}\vert \nabla u(v)\vert ^p\, d\Ha^{\delta}_\infty\Big)^{\frac{1}{p}}\,,
\end{equation*}
where a constant  $c$ depends only on $\delta$, $n$, and $p$.
\end{theorem}

Next we consider functions which belong to $C^1_0$.
We show  that  the classical $L^p$-limiting equality  holds for Choquet integrals case also.

\begin{lemma}\label{lem:Linfty}
Let $\Omega $ be a bounded set  in $\Rn$ and $\delta \in(0, n]$.
Then for every function $f:\Omega \to [-\infty, \infty]$  
\[
\lim_{p \to \infty} \Big( \int_{\Omega}|f|^p\,d\Ha_{\infty}^{\delta}\Big)^{\frac 1p}
= \Ha_{\infty}^{\delta}\text{-}\!\esssup_{x \in \Omega} |f|.
\]
\end{lemma}

\begin{proof}
Let us write $M:= \Ha_{\infty}^{\delta}\text{-}\!\esssup_{x \in \Omega} |f|$. We may assume that $M>0$.
Let $\delta >0$ and let $S_\delta:= \{x \in \Omega : f(x) > M- \delta\}$. Then $\Ha_{\infty}^{\delta}(S_\delta) >0$. Hence we obtain
\[
 \Big(\int_{\Omega}|f|^p\,d\Ha_{\infty}^{\delta}\Big)^{\frac 1p} \ge \big(\Ha_{\infty}^{\delta}(S_\delta)\big)^{\frac 1p} (M-\delta).
\]
Letting  $p \to \infty$  first  and then $\delta \to 0^+$ implies
\[
\liminf_{p \to \infty} \Big( \int_{\Omega}|f|^p\,d\Ha_{\infty}^{\delta}\Big)^{\frac 1p} \ge M.
\]
For a function $f/M$ we obtain
\[
\Big( \int_{\Omega}\Big(\frac{|f|}{M}\Big)^p\,d\Ha_{\infty}^{\delta}\Big)^{\frac 1p}  \le \Big( \int_{\Omega} d\Ha_{\infty}^{\delta}\Big)^{\frac 1p} 
\le \Ha_{\infty}^{\delta}(\Omega)^{\frac1p}.
\]
Letting  $p \to \infty$ yields that
\[
\limsup_{p \to \infty}\Big( \int_{\Omega}\Big(\frac{|f|}{M}\Big)^p\,d\Ha_{\infty}^{\delta}\Big)^{\frac 1p}  \le 1, 
\]
that is  $\limsup_{p \to \infty}\Big( \int_{\Omega}|f|^p\,d\Ha_{\infty}^{\delta}\Big)^{\frac 1p}  \le M$. 
\end{proof}

The following  result gives an embedding from $C^1_0(\Omega)$ to $L^\infty(\Omega, \Ha_{\infty}^{\delta})$.

\begin{theorem}\label{esssupineq}
Let $\Omega$ be a bounded set in $\Rn$.  If  $\delta \in(1, n]$ and  $p\in (\delta,\infty)$, then the inequality
\[
\Ha_{\infty}^{\delta}\text{-}\!\esssup_{x \in \Omega} |u| \le c \Big(\int_\Omega |\nabla u|^p  \, d\Ha_{\infty}^{\delta}\Big)^{\frac1p}
\]
holds
for all $u \in C^1_0(\Omega)$, where the constant $c$ depends only $n$, $\delta$, $p$, and $\Ha_{\infty}^{\delta} (\Omega)$.
\end{theorem}

\begin{proof}
We follow the proof of \cite[Theorem 7.10]{Gilbarg-Trudinger}.
 Let  $\delta \in (0,n)$.
Since $\frac\delta{n}<p=1<\delta$, Theorem~\ref{thm:Poincare-0}
yields
\begin{equation}\label{equ:SP-0}
\Big(\int_\Omega |u|^{\frac{\delta}{\delta-1}}  \, d\Ha_{\infty}^{\delta}\Big)^{\frac{\delta -1}{\delta}}
\le c \int_\Omega |\nabla u|  \, d\Ha_{\infty}^{\delta}
\end{equation}
for all $u \in C^1_0(\Omega)$.
If $\delta =n$, then inequality \eqref{equ:SP-0} is known \cite[Theorem 2.4.1]{Ziemer}.
We apply   this  inequality \eqref{equ:SP-0} to the function  $u^\gamma$ with $\gamma >1$ and use Hölder's inequality 
with $(p, {p}/(p-1)=:p')$ 
to obtain
\[
\begin{split}
\Big(\int_\Omega |u|^{\frac{\gamma\delta }{\delta-1}}  \, d\Ha_{\infty}^{\delta}\Big)^{\frac{\delta -1}{\delta}}
&\le c \gamma \int_\Omega |u|^{\gamma -1}|\nabla u|  \, d\Ha_{\infty}^{\delta}\\
&\le 2c \gamma  \Big(\int_\Omega |u|^{p'(\gamma -1)}  \, d\Ha_{\infty}^{\delta}\Big)^{\frac1{p'}} \Big(\int_\Omega |\nabla u|^p  \, d\Ha_{\infty}^{\delta}\Big)^{\frac1p}. 
\end{split}
\] 
Let  $m:= \max\big\{1, \Ha_{\infty}^{\delta} (\Omega)\big\}$. Applying the above inequality to
\[
\tilde u := \frac{u}{4c m \Big(\int_\Omega |\nabla u|^p  \, d\Ha_{\infty}^{\delta}\Big)^{\frac1p}}
\] 
implies that
\[
\Big(\int_\Omega |\tilde u|^{\frac{\gamma\delta }{\delta-1}}  \, d\Ha_{\infty}^{\delta}\Big)^{\frac{\delta -1}{\delta}}
\le  \frac{\gamma}{2m}  \Big(\int_\Omega |\tilde u|^{p'(\gamma -1)}  \, d\Ha_{\infty}^{\delta}\Big)^{\frac1{p'}}.
\] 
Thus by Hölder's inequality
\[
\begin{split}
\Big(\int_\Omega |\tilde u|^{\frac{\gamma\delta }{\delta-1}}  \, d\Ha_{\infty}^{\delta}\Big)^{\frac{\delta -1}{\gamma\delta}}
&\le  \Big(\frac{\gamma}{2m}\Big)^{\frac1\gamma}  \Big(\int_\Omega |\tilde u|^{p'(\gamma -1)}  \, d\Ha_{\infty}^{\delta}\Big)^{\frac1{\gamma p'}}\\
&\le  \Big(\frac{\gamma}{2m}\Big)^{\frac1\gamma}  \bigg(2 \Big(\int_\Omega |\tilde u|^{p'\gamma}  \, d\Ha_{\infty}^{\delta} \Big)^{\frac{\gamma -1}{\gamma}} \Ha_{\infty}^{\delta} (\Omega)^{\frac1\gamma} \bigg)^{\frac1{\gamma p'}}\\
&=  \gamma^{\frac1\gamma}   \Big( \int_\Omega |\tilde u|^{p'\gamma}  \, d\Ha_{\infty}^{\delta} \Big)^{\frac{\gamma-1}{\gamma^2 p'}}.
\end{split}
\] 
 Let us write $\delta ':=\delta/(\delta -1)$,  $r: = \frac{\delta'}{p'}$. Note that $r>1$, since $p>\delta >1$. We apply the above inequality for $\gamma = r^k$, and iterate it from $1$ to $k$  
and use  \eqref{equ:SP-0}    to obtain
\[
\begin{split}
\Big(\int_\Omega |\tilde u|^{\frac{r^k\delta }{\delta-1}}  \, d\Ha_{\infty}^{\delta}\Big)^{\frac{\delta -1}{r^k\delta}}
\le r^{\sum_{k=1}^\infty k r^{-k}} \Big(\int_\Omega |\tilde u|^{\delta'}  \, d\Ha_{\infty}^{\delta}\Big)^{\frac{\delta'-p'}{\delta'^2}}
\le r^{\sum_{k=1}^\infty k r^{-k}} =: c < \infty\,.
\end{split}
\]
When we let  $k \to \infty$ and use Lemma~\ref{lem:Linfty}, we obtain  that
$\Ha_{\infty}^{\delta}\text{-}\!\esssup_{x \in \Omega} |\tilde u| \le c$. The claim follows.
\end{proof}



\bibliographystyle{amsalpha}

\end{document}